\documentclass{amsart}[12]
\usepackage{mathrsfs, amsmath,amsfonts,amssymb}
\usepackage{amsmath}
\usepackage{mathtools}
\usepackage{fullpage}
\newtheorem{rem}{Remark}

\newtheorem{pro}{Proposition}

\newtheorem{lem}{Lemma}
\title{Hartman-Watson distribution and hyperbolic-like Heat kernels} 
\author[N. Demni]{Nizar Demni}
\address{ Aix-Marseille Universit\'e \\ CNRS\\  Centrale Marseille \\ I2M - UMR 7373\\  39 rue F. Joliot Curie\\  13453 Marseille \\ France }
\email{nizar.demni@univ-amu.fr}
\keywords{Hartman-Watson distribution; Heat kernel; Hyperbolic spaces; Maass Laplacian, Hyperbolic Jacobi operator, Harmonic AN groups. \\ 2010 Mathematics Subject Classification: 22E25; 43A85; 43A90; 60J65.  } 
\begin{document}
\maketitle

\begin{abstract} 
We relate Gruet's formula for the heat kernel on real hyperbolic spaces to the commonly used one derived from Millson induction and distinguishing the parity of the dimensions. The bridge between both formulas is settled by Yor's result on the joint distribution of a Brownian motion and of its exponential functional at fixed time. This result allows further to relate Gruet's formula with real parameter to the heat kernel of the hyperbolic Jacobi operator and to derive a new integral representation for the heat kernel of the Maass Laplacian. When applied to harmonic AN groups (known also as Damek-Ricci spaces), Yor's result yields also a new integral representation of their corresponding heat kernels through the modified Bessel function of the second kind and the Hartman-Watson distribution. This newly obtained formula has the merit to unify both existing formulas in the same way Gruet's formula does for real hyperbolic spaces.
\end{abstract} 
\section{Reminder and Motivation} 
Let  
\begin{equation*}
H_n = \mathbb{R}^{n-1} \times \mathbb{R}_+^{\star}, n \geq 2,
\end{equation*}
be the half-space model of the real hyperbolic space in $n$ dimensions, endowed with its Riemannian metric: 
\begin{equation*}
ds^2 = \frac{dx_1^2+ \dots dx_n^2}{x_n^2}. 
\end{equation*}
The hyperbolic Brownian motion is the diffusion whose generator acts on smooth functions as:
\begin{equation*}
\frac{\Delta_n}{2}:= \frac{1}{2}\left[x_n^2\sum_{i=1}^n\partial_i^2 + (2-n)x_n \partial_n\right]
\end{equation*}
with Neumann conditions on the boundary. Its semi-group density with respect to the Riemannian volume of $H_n$ admits several representations. The commonly used one is derived from Millson induction and reads as follows (see e.g. \cite{Gri-Nog}, \cite{Mat}): 
\begin{itemize}
\item If $n = 2m+1, m \geq 0,$ is odd, then 
\begin{equation}\label{Odd}
q_t^{(2m+1)}(x,y) = q_t^{(2m+1)}(r) = \frac{e^{-m^2t/2}}{(2\pi)^m\sqrt{2\pi t}} \left(-\frac{1}{\sinh(r)}\frac{d}{dr}\right)^me^{-r^2/(2t)},
\end{equation}
where $r = d_n(x,y)$ is the hyperbolic distance between $x, y \in H_n$. 
\item If $n = 2m+2, m \geq 0,$ is even, then  
\begin{equation}\label{Even}
q_t^{(2m+2)}(x,y) = q_t^{(2m+2)}(r) = \frac{e^{-(2m+1)^2t/8}\sqrt{2}}{(2\pi t)^{3/2} (2\pi)^m} \left(-\frac{1}{\sinh(r)}\frac{d}{dr}\right)^m \int_r^{\infty} \frac{\theta e^{-\theta^2/(2t)}}{(\cosh(\theta) - \cosh(r))^{1/2}} d\theta.
\end{equation}\end{itemize}
Another less-known formula for the hyperbolic heat semi-group was derived in \cite{Gruet} and does not distinguish odd and even dimensions. It is given by the following oscillatory integral (see also \cite{Mat}):
\begin{equation}\label{GruetFor}
q_t^n(x,y) = q_t^n(r) = \frac{e^{-(n-1)^2t/8}}{\pi(2\pi)^{n/2}\sqrt{t}}\Gamma\left(\frac{n+1}{2}\right)\int_0^{\infty}\frac{e^{(\pi^2-\rho^2)/(2t)}\sinh(\rho)\sin(\pi\rho/t)}{[\cosh(\rho) + \cosh(r)]^{(n+1)/2}} d\rho,
\end{equation}
and is obtained after appealing to the conditional distribution of the exponential functional 
\begin{equation*}
A_t:= \int_0^te^{2B_s}ds, \quad t > 0,
\end{equation*}
given $B_t$, where $(B_s)_{s \geq 0}$ is a linear Brownian motion. Actually, M. Yor proved that the conditional distribution of $A_t$ given $B_t = \theta$ admits the following expression: 
\begin{equation}\label{Yor} 
\mathbb{P}(A_t \in ds| B_t = \theta) := a_t(s,\theta)ds :=  \frac{1}{s}e^{-(1+e^{2\theta})/(2s)} u\left(t, \frac{e^{\theta}}{s}\right)ds,
\end{equation}
where 
\begin{equation}\label{HW}
u(t,y) := \frac{y}{\pi\sqrt{2\pi t}}\int_0^{\infty}e^{(\pi^2-\rho^2)/(2t)}e^{-y\cosh(\rho)}\sinh(\rho)\sin\left(\frac{\pi\rho}{t}\right) d\rho. 
\end{equation}
Up to a normalising factor, $u(t,y)$ is the density of the so-called Hartman-Watson distribution at time $t$ (\cite{Har-Wat}). A recent and deep study of this density may be found in \cite{Jak-Wis} where new interpretations through real Brownian motion were obtained. 

In \cite{Gruet}, Gruet raised the problem of relating his formula \eqref{GruetFor} to \eqref{Odd} and \eqref{Even} and answered this question there when $n \in \{2,3\}$. He also asked whether \eqref{GruetFor} still gives the heat kernel of the radial part of 
$\Delta_n/2$ when $n > 1$ is a positive real parameter, which is an instance of the hyperbolic Jacobi operator recalled below (\cite{Koo}). This question was subsequently answered by Gruet himself in \cite{Gruet1}. The proofs written in \cite{Gruet} and \cite{Gruet1} rely on Cauchy Residue's Theorem and on the Green function of the hyperbolic Jacobi operator. In this paper, we rather answer Gruet questions alluded to above by simply observing that Yor's result \eqref{Yor} entails the following identity (multiply \eqref{Yor} by the Gaussian density then integrate over $s$): 
\begin{equation}\label{Int1}
\frac{1}{\sqrt{2\pi t}} e^{-\theta^2/(2t)} = \int_0^{\infty} e^{-\cosh(\theta) y} \frac{u(t,y)}{y} dy, \quad t > 0, \theta \in \mathbb{R}.
\end{equation}
Doing so provides simpler proofs of \eqref{GruetFor} and of its connection to the hyperbolic Jacobi operator, making only use of techniques from real analysis. The identity \eqref{Int1} applies further to harmonic AN groups (known also as Damek-Ricci spaces) in which case it yields a new integral representation for their heat kernels through the modified Bessel function of the second kind and the density $u(t,y)$. This newly obtained representation does not distinguish the parity of the center of $N$ as Gruet's formula does for real hyperbolic spaces. In particular, we retrieve the integral representations of the heat kernels of complex and quaternionic hyperbolic spaces derived in \cite{Mat} using Malliavin calculus. 

Another interesting application of \eqref{Int1} is concerned with the Maass Laplacian. This operator originates in analytic number theory where it is used in the study of the so-called Maass forms (\cite{Aya-Int}) and in mathematical physics as well where it quantizes the energy of a particle in the hyperbolic half-plane submitted to a uniform magnetic field (\cite{Com}). Its heat kernel admits a quite similar integral representation to \eqref{Even} which comes with an additional factor in the integrand. The latter makes the application of \eqref{Int1} to the Maass heat kernel different from those corresponding to hyperbolic spaces due to integrability issues. In particular, we obtain an integral representation through the Hartman-Watson density and a generalised Bessel polynomial. 

The paper is organised as follows. The next section is devoted to the derivation of Gruet's formula \eqref{GruetFor} from \eqref{Int1} and to its analogue for the heat kernel of the Maass Laplacian. In section 3, we prove that the heat kernel of the hyperbolic Jacobi operator is still given by \eqref{GruetFor} with $n > 1$ being a real number. In the last section, we apply \eqref{Int1} to derive our new integral representation of heat kernels of $AN$ groups and retrieve their special instances corresponding to complex and quaternionic hyperbolic spaces. 

\section{Real hyperbolic spaces: another proof of Gruet's formula}
In order to relate \eqref{GruetFor} to \eqref{Odd} and \eqref{Even}, we need the following lemma: 
\begin{lem}
For any $t > 0, \theta \in \mathbb{R}$, 
\begin{equation}\label{Int2}
e^{-\theta^2/(2t)} = \frac{1}{\pi} \int_0^{\infty} e^{(\pi^2-\rho^2)/(2t)}\sinh(\rho)\sin\left(\frac{\pi\rho}{t}\right) \frac{d\rho}{\cosh(\rho) + \cosh(\theta)}, 
\end{equation}
and in turn 
\begin{equation}\label{Int3}
\theta e^{-\theta^2/(2t)} = \frac{t\sinh(\theta)}{\pi} \int_0^{\infty} e^{(\pi^2-\rho^2)/(2t)}\sinh(\rho)\sin\left(\frac{\pi\rho}{t}\right) \frac{d\rho}{[\cosh(\rho) + \cosh(\theta)]^2}.
\end{equation}
\end{lem} 
\begin{proof}
Substitute \eqref{HW} in \eqref{Int1} and use Fubini Theorem to obtain \eqref{Int2}. Then differentiate the latter with respect to $\theta$ to get \eqref{Int3}.
\end{proof}
Now, obtaining \eqref{GruetFor} from \eqref{Odd} is straightforward. Indeed, assume $n = 2m+1$ is odd then \eqref{Int2} entails: 
\begin{equation*}
\left(-\frac{1}{\sinh(r)}\frac{d}{dr}\right)^me^{-r^2/(2t)} = \frac{m!}{\pi} \int_0^{\infty} e^{(\pi^2-\rho^2)/(2t)}\sinh(\rho)\sin\left(\frac{\pi\rho}{t}\right) \frac{d\rho}{[\cosh(\rho) + \cosh(r)]^{m+1}}. 
\end{equation*}
As to the derivation of \eqref{GruetFor} from \eqref{Even}, it goes as follows. Use \eqref{Int3} and Fubini Theorem to get: 
\begin{align*}
\int_r^{\infty} \frac{\theta e^{-\theta^2/(2t)}}{(\cosh(\theta) - \cosh(r))^{1/2}} d\theta & = \frac{t}{\pi} \int_0^{\infty} e^{(\pi^2-\rho^2)/(2t)}\sinh(\rho)\sin\left(\frac{\pi\rho}{t}\right) d\rho 
\\& \int_r^{\infty}\frac{\sinh(\theta)}{(\cosh(\theta) - \cosh(r))^{1/2}[\cosh(\rho) + \cosh(\theta)]^2}d\theta 
\\& = \frac{t}{\pi} \int_0^{\infty} e^{(\pi^2-\rho^2)/(2t)}\sinh(\rho)\sin\left(\frac{\pi\rho}{t}\right) d\rho  \int_0^{\infty}\frac{du}{u^{1/2}[u+\cosh(r)+\cosh(\rho)]^2}d\theta 
\\&  = \frac{t}{2} \int_0^{\infty} e^{(\pi^2-\rho^2)/(2t)}\sinh(\rho)\sin\left(\frac{\pi\rho}{t}\right)\frac{d\rho}{[\cosh(\rho)+\cosh(r)]^{3/2}},
\end{align*}
where the last equality follows from the Beta prime integral: 
\begin{equation}\label{BP}
\int_0^{\infty} \frac{v^{a-1}}{(1+v)^{a+b}}dv = \frac{\Gamma(a)\Gamma(b)}{\Gamma(a+b)}, \quad a,b >0.
\end{equation}
Together with Legendre duplication formula: 
\begin{equation*}
\sqrt{\pi} \Gamma(2z) = 2^{2z-1}\Gamma(z)\Gamma\left(z+\frac{1}{2}\right), 
\end{equation*}
we end up with:
\begin{multline*}
\left(-\frac{1}{\sinh(r)}\frac{d}{dr}\right)^m \int_r^{\infty} \frac{\theta e^{-\theta^2/(2t)}}{(\cosh(\theta) - \cosh(r))^{1/2}} d\theta = \frac{t}{2} \int_0^{\infty} e^{(\pi^2-\rho^2)/(2t)}\sinh(\rho)\sin\left(\frac{\pi\rho}{t}\right)
\\  \frac{3}{2}\dots \frac{2m+1}{2} \frac{d\rho}{[\cosh(\rho)+\cosh(r)]^{m+3/2}} =  t\frac{\Gamma(2m+2)}{2^{2m+1}\Gamma(m+1)} \\ 
 \int_0^{\infty} e^{(\pi^2-\rho^2)/(2t)}\sinh(\rho)\sin\left(\frac{\pi\rho}{t}\right)\frac{d\rho}{[\cosh(\rho)+\cosh(r)]^{m+3/2}} \\ 
= \frac{t}{\sqrt{\pi}} \Gamma\left(m+\frac{3}{2}\right) \int_0^{\infty} e^{(\pi^2-\rho^2)/(2t)}\sinh(\rho)\sin\left(\frac{\pi\rho}{t}\right)\frac{d\rho}{[\cosh(\rho)+\cosh(r)]^{m+3/2}}.
\end{multline*}
Keeping in mind the factor 
\begin{equation*}
\frac{e^{-(2m+1)^2t/8}\sqrt{2}}{(2\pi t)^{3/2} (2\pi)^m} = \frac{e^{-(2m+1)^2t/8}}{2^{m+1}t^{3/2} \pi^{m+3/2}}
\end{equation*} 
in \eqref{Even}, Gruet's formula \eqref{GruetFor} with $n=2m+2$ follows. 

\subsection{Application to the Maass Laplacian}
Let $\mathbb{H} = H_2$ is the Poincar\'e upper-half-plane. Then the Maass Laplacian is defined by (\cite{AMS}, \cite{Ike-Mat}):
\begin{equation*}
\mathscr{L}_k := -\frac{y^2}{2}(\partial_w^2+\partial_y^2) + iky\partial_w + \frac{k^2}{2}, \quad w+iy \in \mathbb{H},
\end{equation*}
where $k \in \mathbb{R}$. This operator arises in number theory as well as in mathematical physics and is, when $k \in \mathbb{Z}$, the AN part of the Laplace Beltrami-operator of $Sl(2, \mathbb{R})$ in Iwasawa coordinates . 
At the spectral level, $\mathscr{L}_k$ is a densely-defined and essentially-self-adjoint operator in 
\begin{equation*}
L^2\left(\mathbb{H}, dw \frac{dy}{y^2}\right),
\end{equation*}
and the heat kernel of its self-adjoint closure with initial point $z=i$ admits the following expression (\cite{AMS}, \cite{Ike-Mat}): 
\begin{multline}\label{MaassSD}
q_{t,k}^{(2)}(i,w+iy) = \left(\frac{w+i(y+1)}{i(y+1)-\overline{w}}\right)^k\frac{\sqrt{2}e^{-t/8-k^2t/2}}{(2\pi t)^{3/2}} \\ 
\int_{r}^{\infty}d\theta\frac{\theta e^{-\theta^2/(2t)}}{\sqrt{\cosh(\theta) - \cosh(r)}} \cosh\left\{2k\cosh^{-1}\left(\frac{\cosh(\theta/2)}{\cosh(r/2)}\right)\right\}
\end{multline}
with respect to the hyperbolic volume measure $dw dy/y^2$. Here the principal determination of the power is taken and we still denote $r = d_2(i,w+iy)$ the hyperbolic distance. In particular, $\mathscr{L}_0 = \Delta_2/2$ and $q_{t,0}^{(2)}(i,w+iy)$ reduces to the special instance $m=0$ of \eqref{Even}. 
If we substitute \eqref{Int3} in the RHS of \eqref{MaassSD}, then Fubini Theorem no longer applies since the resulting double integral does not converge absolutely, except for small values of $|k|$. Nonetheless, we can use \eqref{Int1} to prove:
\begin{pro} 
If $k$ is positive integer, then for any $t > 0$,
\begin{equation*}
q_{t,k}^{(2)}(i,w+iy) = \left(\frac{w+i(y+1)}{i(y+1)-\overline{w}}\right)^k\frac{e^{-t/8-k^2t/2}}{\sqrt{\pi}}  \int_0^{\infty} dy \frac{u(t,z)}{\sqrt{z}} {}_2F_0\left(-k,k; -\frac{1}{z(1+\cosh(r))}\right),
\end{equation*}
where 
\begin{equation*}
 {}_2F_0\left(-k,k; z\right) = \sum_{j=0}^{k} \frac{(-k)_j(k)_j}{j!} z^j, \quad z \in \mathbb{R},
 \end{equation*}
 is a generalised Bessel polynomial (\cite{Rai}). 
\end{pro}
\begin{rem}
The  integral derived in the previous proposition converges absolutely since $u(t,z) = O(z^N)$ for any positive integer $N$ (\cite{Mat-Yor}, p. 188).  Besides, the proof below remains valid for real $k$. 
\end{rem}
\begin{proof}
Differentiate \eqref{Int1} with respect to $\theta$ and substitute the resulting identity in \eqref{MaassSD}. Then, perform the variable change $v = \cosh(\theta)$ in the  double integral to get:  
\begin{multline*}
q_{t,k}^{(2)}(i,w+iy) = \left(\frac{w+i(y+1)}{i(y+1)-\overline{w}}\right)^k\frac{e^{-t/8-k^2t/2}}{\sqrt{2}\pi} \\ \int_0^{\infty} dy u(t,z) \int_{\cosh(r)}^{\infty}\frac{dv e^{-zv}}{\sqrt{v - \cosh(r)}}
\cosh\left\{2k\cosh^{-1}\left(\sqrt{\frac{1+v}{1+\cosh(r)}}\right)\right\}.
\end{multline*}
Now, observe that 
\begin{equation*}
\cosh\left\{2k\cosh^{-1}\left(\sqrt{\frac{1+v}{1+\cosh(r)}}\right)\right\} = T_{2k}\left(\sqrt{\frac{1+v}{1+\cosh(r)}}\right),
 \end{equation*}
 where
\begin{equation*}
T_{2k}(z) = {}_2F_1\left(-2k, 2k, \frac{1}{2}, \frac{1-z}{2}\right)
\end{equation*}
is the $2k$-th Tchebycheff polynomial (${}_2F_1$ being the Gauss hypergeometric function). Using the quadratic transformation (\cite{Erd1}, p.111, (2)):
\begin{equation*}
{}_2F_1\left(-2k, 2k, \frac{1}{2}, \frac{1-z}{2}\right) = {}_2F_1\left(-k, k, \frac{1}{2}, 1-z^2\right), 
\end{equation*}
which is equivalent to the identity $T_{2k}(z) = T_k(2z^2-1)$, we can further write (we also perform the variable change $v \rightarrow v + \cosh(r)$):
\begin{multline*}
q_{t,k}^{(2)}(i,w+iy) = \left(\frac{w+i(y+1)}{i(y+1)-\overline{w}}\right)^k\frac{e^{-t/8-k^2t/2}}{\sqrt{2}\pi} \int_0^{\infty} dy u(t,z) e^{-z\cosh(r)} \int_{0}^{\infty}\frac{dve^{-zv}}{\sqrt{v}} \\ {}_2F_1\left(-k, k, \frac{1}{2}, -\frac{v}{1+\cosh(r)}\right)
 = \left(\frac{w+i(y+1)}{i(y+1)-\overline{w}}\right)^k\frac{e^{-t/8-k^2t/2}}{\pi}\cosh(r/2)  \int_0^{\infty} dy \frac{u(t,z)}{\sqrt{z}} \\  \int_{0}^{\infty}dv\frac{e^{-(1+\cosh(r))zv}}{\sqrt{v}}{}_2F_1\left(-k, k, \frac{1}{2}, -v\right).
\end{multline*}
Finally, expanding the Gauss hypergeometric function: 
\begin{equation*}
{}_2F_1\left(-k, k, \frac{1}{2}, -v\right) = \sqrt{\pi} \sum_{j=0}^k\frac{(-k)_k(k)_j}{\Gamma(j+1/2) j!} (-v)^j,
\end{equation*}
and integrating term-wise the inner integral, the proposition follows.  
\end{proof}

\section{The Interpolated heat kernel: another proof of Gruet's conjecture}
In the last section of \cite{Gruet}, the author conjectured that his formula \eqref{GruetFor} still gives, when $n = 2\nu+2 > 1$ is real, the heat kernel of the hyperbolic Jacobi operator:
\begin{equation*}
\frac{1}{2}\left[\partial_r^2 + (2\nu+1)\coth(r)\partial_r\right], \quad r > 0,
\end{equation*}
with Neumann boundary condition at $r=0$. Using Cauchy integral formula, Gruet proved his conjecture in \cite{Gruet} when the parameter $\nu$ ranges in $]-1/2, 1/2[$ and subsequently proved it for all parameters $\nu > -1/2$ in \cite{Gruet1} (see Appendix there). In this paragraph, we give another proof of this conjecture relying on \eqref{Int2} and standard facts from hyperbolic Jacobi analysis (\cite{Koo}). To this end, we start from the spectral formula for the hyperbolic Jacobi heat kernel (see e.g. \cite{Gruet}): 
\begin{equation}\label{HeatJacobi}
p_t^{(\nu)}(0,r) = \frac{e^{-(2\nu+1)^2t/2}}{\pi^{\nu+1}2^{2\nu+1}\Gamma(\nu+1)}\int_0^{\infty}e^{-p^2t/2} \phi_p(r) \left|\frac{\Gamma(ip+\nu+1/2)}{\Gamma(ip)}\right|^2 dp,
\end{equation}
with respect to the radial volume element:
\begin{equation*}
2\frac{\sinh^{2\nu+1}(r)\pi^{\nu+1}}{\Gamma(\nu+1)}.
\end{equation*}
Here 
\begin{align*}
\phi_p(r) & = {}_2F_1\left(\frac{\nu + 1/2 -ip}{2}, \frac{\nu + 1/2 +ip}{2}, \nu+1; -\sinh^2(r)\right)
\end{align*}
is a Jacobi function and is an eigenfunction of the hyperbolic Jacobi operator. 
\begin{pro}\label{GCI}
For any real $n = 2\nu+2 > 1$, \eqref{GruetFor} gives the heat kernel of the hyperbolic Jacobi operator.  
\end{pro}
\begin{proof}
Using the Fourier transform: 
\begin{equation*}
\int_{\mathbb{R}} e^{ip\theta} e^{-\theta^2/(2t)} \frac{d\theta}{\sqrt{2\pi t}} = e^{-p^2t/2},
\end{equation*}
together with \eqref{Int2}, it follows that: 
\begin{align*}
e^{-p^2t/2} & = \frac{1}{\pi\sqrt{2\pi t} } \int_0^{\infty} d\rho e^{(\pi^2-\rho^2)/(2t)}\sinh(\rho)\sin\left(\frac{\pi\rho}{t}\right) \int_{\mathbb{R}} e^{ip\theta}\frac{d\theta}{\cosh(\rho) + \cosh(\theta)}
\\& = \frac{1}{\pi\sqrt{2\pi t} } \int_0^{\infty} d\rho e^{(\pi^2-\rho^2)/(2t)}\sinh(\rho)\sin\left(\frac{\pi\rho}{t}\right) \mathscr{F}\left(\theta \mapsto \frac{1}{\cosh(\rho) + \cosh(\theta)}\right)(p)
\end{align*}
where $\mathscr{F}$ denotes the Euclidean Fourier transform. Now, recall from \cite{Koo}, Theorem 2.3, the inverse Jacobi transform of a rapidly-decreasing function $f: \mathbb{R}_+ \mapsto \mathbb{R}$:
\begin{equation*}
J_{\nu}^{-1}: f \mapsto J_{\nu}^{-1}(f): r \mapsto \frac{1}{2^{4\nu+1}[\Gamma(\nu+1)]^2}\int_0^{\infty} f(p) \phi_p(r) \left|\frac{\Gamma(ip+\nu+1/2)}{\Gamma(ip)}\right|^2 dp.
\end{equation*}
Consequently, \eqref{HeatJacobi} may be written as: 
\begin{multline*}
p_t^{(\nu)}(0,r) = \frac{2^{2\nu}\Gamma(\nu+1)e^{-(2\nu+1)^2t/2}}{\pi^{\nu+2}\sqrt{2\pi t}}
\int_0^{\infty} d\rho e^{(\pi^2-\rho^2)/(2t)}\sin\left(\frac{\pi\rho}{t}\right)\sinh(\rho) \\ J_{\nu}^{-1}\left\{\mathscr{F}\left(\theta \mapsto \frac{1}{\cosh(\rho) + \cosh(\theta)}\right)\right\}(r),
\end{multline*}
But, the Jacobi transform $J_{\nu}$ is the composition of the Jacobi-Abel transform $A_{\nu}$ and of the Euclidean Fourrier transform $\mathscr{F}$ (\cite{Koo}, Section 5): 
\begin{equation*}
J_{\nu} = \mathscr{F} \circ A_{\nu}. 
\end{equation*}
As a matter of fact, 
\begin{multline*}
p_t^{(\nu)}(0,r) = \frac{2^{2\nu}\Gamma(\nu+1)e^{-(2\nu+1)^2t/2}}{\pi^{\nu+2}\sqrt{2\pi t}}
\int_0^{\infty} d\rho e^{(\pi^2-\rho^2)/(2t)}\sin\left(\frac{\pi\rho}{t}\right)\sinh(\rho) \\ A_{\nu}^{-1}\left(\theta \mapsto \frac{1}{\cosh(\rho) + \cosh(\theta)}\right)(r),
\end{multline*}
where $A_{\nu}^{-1}$ stands for the inverse Jacobi-Abel transform. Moreover, we infer from \cite{Koo}, (5.63), (5.64), that: 
\begin{multline*}
A_{\nu}^{-1}(f)(r)  = \frac{(-1)^{[\nu+1/2]+1}\sqrt{\pi}}{2^{3\nu+1/2}\Gamma(\nu+1)\Gamma([\nu+1/2] +(1/2)-\nu)} \\ 
\int_r^{\infty} \left[\frac{1}{\sinh(\theta)}\frac{df}{d\theta}\right]^{[\nu+1/2]+1}(\theta) (\cosh(\theta) - \cosh(r))^{[\nu+1/2]-(\nu + 1/2)}\sinh(\theta) d\theta,
\end{multline*}
where $[\nu+1/2]$ is the integral part of $\nu+1/2$.
Finally,
\begin{multline*}
p_t^{(\nu)}(0,r) = \frac{e^{-(2\nu+1)^2t/2}\Gamma([\nu+1/2]+2)}{\pi^{\nu+2}2^{\nu+1}\sqrt{t}\Gamma([\nu+1/2] +(1/2)-\nu)} \int_0^{\infty} e^{(\pi^2-\rho^2)/(2t)}\sin\left(\frac{\pi\rho}{t}\right)\sinh(\rho) \\ 
\int_r^{\infty} \frac{(\cosh(\theta) - \cosh(r))^{[\nu+1/2]-(\nu + 1/2)}}{[\cosh(\rho) + \cosh(\theta)]^{[\nu+1/2]+2}}\sinh(\theta)d\theta
\end{multline*}
and Gruet's formula (with $n = 2\nu+2$) is obtained after performing the variable change 
\begin{equation*}
v= \frac{\cosh(\theta) - \cosh(r)}{\cosh(\rho)+\cosh(r)}
\end{equation*}
 in the inner integral and using the Beta prime integral \eqref{BP}.
\end{proof}

\section{Application to Harmonic $AN$ groups}
Harmonic AN groups, known also as Damek-Ricci spaces, are solvable extensions of H-type groups introduced and studied by Kaplan (\cite{Kap}, \cite{Dam-Ric}, \cite{ADY}). They include all rank-one symmetric spaces of non compact type (real hyperbolic spaces are disregarded since $N$ is abelian), yet most of them are not symmetric. Moreover, any such group is the semi-direct product of a one dimensional group $A \sim \mathbb{R}_+^{\star}$ and a H-type groups $N$. 

The radial part of the Laplace-Beltrami operator in a harmonic AN group reduces to the following hyperbolic Jacobi operator: 
\begin{equation}\label{HyperJacOp}
\frac{1}{2}\left\{\partial_r^2 + \left[\frac{m+k}{2}\coth\left(\frac{r}{2}\right) + \frac{k}{2}\tanh\left(\frac{r}{2}\right)\right]\partial_r\right\}. 
\end{equation}
Here, $r$ is the geodesic distance with respect to the origin in the Ball realisation of the group and $(k,m)$ are respectively the positive dimensions of the center and of its ortho-complement  in the Lie algebra of N ($m$ is even, \cite{ADY}, p.644). From the general framework of hyperbolic Jacobi analysis, the heat kernel admits the following expressions (\cite{ADY}, p.663). Let $M := m+k+1$ and $Q:= k+(m/2)$, then\footnote{We use the probabilistic normalisation for the heat equation which amounts to substitute $t \rightarrow t/2$.}
\begin{itemize}
\item If $k$ is even then 
\begin{equation}\label{EvenAN}
q_t^{(k,m)}(r) = \frac{e^{-Q^2t/8}}{2^{m+1+k/2} \pi^{M/2}\sqrt{t}} \left(-\frac{1}{\sinh(r)}\frac{d}{dr}\right)^{k/2}  \left(-\frac{1}{\sinh(r/2)}\frac{d}{dr}\right)^{m/2} e^{-r^2/(2t)},
\end{equation}
 \item  If $k$ is odd then
\begin{multline}\label{OddAN}
 q_t^{(k,m)}(r) =  \frac{e^{-Q^2t/8}}{2^{m+1+k/2} \pi^{(M+1)/2}\sqrt{t}} \int_r^{\infty} \frac{\sinh(\theta) d\theta}{(\cosh(\theta) - \cosh(r))^{1/2}} 
 \\ \left(-\frac{1}{\sinh(\theta)}\frac{d}{d\theta }\right)^{(k+1)/2} \left(-\frac{1}{\sinh(\theta/2)}\frac{d}{d\theta }\right)^{m/2}e^{-\theta^2/(2t)}.
\end{multline}
\end{itemize}

Using \eqref{Int1}, we derive a new integral representation of $q_t^{(k,m)}(r)$ which does not distinguish the parity of $k$: 
\begin{pro}\label{HeatNA}
For any $t > 0$, 
\begin{equation*}
q_t^{(k,m)}(r) = \frac{e^{-Q^2t/8}}{2^{3(m+k)/2+1} \pi^{M/2}[\cosh(r/2)]^{(k-1)/2}} \int_0^{\infty} y^{(m+k-1)/2} u\left(\frac{t}{4}, y\right) K_{(k-1)/2}\left(y\cosh\left(\frac{r}{2}\right)\right) dy, 
\end{equation*}
where $K_{\nu}$ is the modified Bessel function of the second kind. 
\end{pro}
\begin{proof}
Substituting $\theta \rightarrow r/2, t \rightarrow t/4$ in \eqref{Int1}, we get 
\begin{equation*}
e^{-r^2/(2t)} = \sqrt{\frac{\pi t}{2}} \int_0^{\infty} e^{-\cosh(r/2) y} u\left(\frac{t}{4}, y\right)  \frac{dy}{y},
\end{equation*}
whence it readily follows:
\begin{align*}
\left(-\frac{1}{\sinh(r/2)}\frac{d}{dr}\right)^{m/2} e^{-r^2/(2t)} = \frac{\sqrt{\pi t/2}}{2^{m/2}} \int_0^{\infty}y^{(m/2)-1} e^{-\cosh(r/2) y} u\left(\frac{t}{4}, y\right) dy.
\end{align*}
Consequently, if $k \geq 2$ is even then the formula $\sinh(r) = 2\sinh(r/2)\cosh(r/2)$ entails:
\begin{multline*}
\left(-\frac{1}{\sinh(r)}\frac{d}{dr}\right)^{k/2}  \left(-\frac{1}{\sinh(r/2)}\frac{d}{dr}\right)^{m/2} e^{-r^2/(2t)} = \frac{\sqrt{\pi t/2}}{2^{k+(m/2)}} \int_0^{\infty}y^{(m/2)} u\left(\frac{t}{4}, y\right) \\
 \left(-\frac{1}{v}\frac{d}{dv}\right)^{k/2-1}\left(\frac{e^{-vy}}{v}\right)_{v = \cosh(r/2)} dy.
\end{multline*}
Now, noting that 
\begin{equation*}
-\frac{1}{v}\frac{d}{dv}\frac{e^{-vy}}{v} = -\frac{y^2}{v}\left(\frac{d}{dv} \frac{e^{-v}}{v}\right) (vy) = y^3 \left(-\frac{1}{v}\frac{d}{dv}\frac{e^{-v}}{v}\right)(vy), 
\end{equation*}
we get 
\begin{align*}
\left(-\frac{1}{v}\frac{d}{dv}\right)^{k/2-1}\left(\frac{e^{-vy}}{v}\right)(v) = y^{k-1} \left\{\left(-\frac{1}{v}\frac{d}{dv}\right)^{(k/2)-1}\frac{e^{-v}}{v}\right\}(vy).
\end{align*}
But the modified Bessel function of the second kind $K_{n+1/2}$ admits the following representation (\cite{Erd2}, 7.2.6, (43)): 
\begin{equation*}
 K_{n+1/2}(v) = \sqrt{\frac{\pi}{2}} v^{n+1/2} \left(-\frac{1}{v}\frac{d}{dv}\right)^{n}\frac{e^{-v}}{v}. 
 \end{equation*}
As a result: 
\begin{multline*}
\left(-\frac{1}{\sinh(r)}\frac{d}{dr}\right)^{k/2}  \left(-\frac{1}{\sinh(r/2)}\frac{d}{dr}\right)^{m/2} e^{-r^2/(2t)} =   \frac{\sqrt{t}}{2^{k+(m/2)}[\cosh(r/2)]^{(k-1)/2}}\\  
\int_0^{\infty} y^{(m+k-1/2)} u\left(\frac{t}{4}, y\right) K_{(k-1)/2}(y\cosh(r/2)) dy.
\end{multline*}
Keeping in mind \eqref{EvenAN}, we get the sought integral representation. 

If $k$ is odd then \eqref{OddAN} is similarly transformed into: 
\begin{multline*}
q_t^{(k,m)}(r) =  \frac{e^{-Q^2t/8}}{2^{3(m+k)/2 +2} \pi^{(M+1)/2}} \int_r^{\infty} \frac{\sinh(\theta) d\theta}{(\cosh(\theta) - \cosh(r))^{1/2}[\cosh(\theta/2)]^{k/2}} \\ 
\int_0^{\infty} y^{(m+k)/2} u\left(\frac{t}{4}, y\right) K_{k/2}(y\cosh(\theta/2)) dy, 
\end{multline*}
which may be further written as: 
\begin{multline*}
q_t^{(k,m)}(r) =  \frac{2^{k/4}e^{-Q^2t/8}}{2^{3(m+k)/2 +2} \pi^{(M+1)/2}} \int_0^{\infty}dy  y^{(m+k)/2} u\left(\frac{t}{4}, y\right) \\ 
\int_0^{\infty} \frac{dw}{\sqrt{w}[\sqrt{1+\cosh(r) + w}]^{k/2}} K_{k/2}\left(\frac{y\sqrt{1+\cosh(r) + w}}{\sqrt{2}}\right). 
\end{multline*}
Finally, the variable change $w \mapsto (1+\cosh(r)) w$ transforms the inner integral into: 
\begin{multline*}
\frac{1}{(1+\cosh(r))^{(k-2)/4}}\int_0^{\infty} \frac{dw}{\sqrt{w}[\sqrt{1+ w}]^{k/2}} K_{k/2}\left(y\cosh(r/2) \sqrt{1+ w}\right) = \frac{1}{2^{(k-2)/4}[\cosh(r/2)]^{(k-2)/2}} 
\\ \int_1^{\infty} \frac{dw}{\sqrt{w-1}w^{k/4}} K_{k/2}\left(y\cosh(r/2)\sqrt{w}\right) = \frac{\sqrt{2\pi}}{2^{(k-2)/4}\sqrt{y}[\cosh(r/2)]^{(k-1)/2}} K_{(k-1)/2}\left(y\cosh(r/2)\right),
\end{multline*}
where the last equality follows from the identity (\cite{Gra-Ryz}, p. 691, (12)): 
\begin{equation*}
\int_1^{\infty} w^{-a/2}(w-1)^{b-1} K_{a}(x\sqrt{w}) dw = \frac{\Gamma(b)2^b}{x^b}K_{a-b}(x), \quad a,b > 0, x \in \mathbb{R}.  
\end{equation*}
Altogether, we get the same integral representation for the heat kernel. The proposition is proved. 
\end{proof} 

\begin{rem}
We can mimic the proof of Proposition \ref{GCI} and use similar computations written in the proof of Proposition \ref{HeatNA} to extend the previous integral representation of $q_t^{(k,m)}(r)$ to heat kernel of the interpolated hyperbolic Jacobi operator: 
\begin{equation*}
\frac{1}{2}\left\{\partial_r^2 + \left[(2a+1)\coth\left(\frac{r}{2}\right) + (2b+1)\tanh\left(\frac{r}{2}\right)\right]\partial_r\right\}
\end{equation*}
provided $a > b > 1/2$. 
\end{rem}

\subsection{Real hyperbolic spaces}
Real hyperbolic spaces $H_n(\mathbb{R})$ are $AN$ parts in the Iwasawa decomposition of the real Lie group $SO_0(n,1)$. Since $N \sim \mathbb{R}^{n-1}$ is abelian, they are often disregarded from the family of Damek-Ricci spaces. Nonetheless, we can assign to them the parameters $k=0, m=n-1$ (note that $m$ may take odd values in this case). If we allow $k=0$ and even values of $m$ in the first formula of Proposition \ref{HeatNA} and if we use the identity 
\begin{equation*}
K_{-1/2}(z) = \sqrt{\frac{\pi}{2z}} e^{-z}, \quad z > 0,
\end{equation*}
then we obtain:
\begin{equation*}
q_t^{(0,m)}(r) = \frac{e^{-m^2t/32}}{2^{3m/2+1} \pi^{(m+1)/2}} \int_0^{\infty} y^{m/2} u\left(\frac{t}{4}, y\right) e^{-y\cosh(r/2)} dy.
\end{equation*}
Appealing further to \eqref{HW}, to Fubini Theorem and to the Gamma integral, we get \eqref{GruetFor} up to normalisations.   

\subsection{Complex and quaternionic hyperbolic spaces}
The heat kernel of the complex hyperbolic space $H_n(\mathbb{C})$ admits the following expression (\cite{Mat}, Theorem 4.1, \cite{Ould}): 
\begin{equation}\label{CH1}
h_t^{(n)}(r) = \frac{2e^{-n^2t/2}}{(2\pi t)^{1/2} (2\pi)^n}  \int_r^{\infty}\frac{\sinh(\theta)}{(\cosh^2(\theta) - \cosh^2(r))^{1/2}} \left(-\frac{1}{\sinh(\theta)}\frac{d}{d\theta}\right)^ne^{-\theta^2/(2t)} d\theta.
\end{equation}
Using Malliavin calculus and \eqref{GruetFor}, the following integral representation was derived in \cite{Mat} (see the bottom of p. 570):
\begin{multline}\label{CH2}
h_t^{(n)}(r) = \frac{4n!e^{-n^2t/2}}{(2\pi t)^{1/2} (2\pi)^{n+1}} \int_0^{\infty} e^{(\pi^2-\rho^2)/(2t)}\sinh(\rho)\sin\left(\frac{\pi\rho}{t}\right)  d\rho \int_0^{\infty}\frac{dv}{[\cosh(r)\cosh(v) + \cosh(\rho)]^{n+1}} dv.
\end{multline}
 On the other hand, the complex hyperbolic space is a harmonic AN group corresponding $k=1, m = 2(n-1)$ (\cite{ADY}, p. 645). As a matter of fact, our formula reads:
\begin{equation}\label{CompHyp}
q_t^{(1,2(n-1))}(r)  = \frac{e^{-n^2t/8}}{2^{3(2n-1)/2+1} \pi^{n}} \int_0^{\infty}  y^{n-1} u\left(\frac{t}{4}, y\right)K_{0}\left(y\cosh(r/2)\right) dy.
\end{equation}
Up to the normalisations $t \mapsto t/4$ and $r \mapsto r/2$, \eqref{CompHyp} may be related to \eqref{CH2} by simply noting that
\begin{align*}
n! \int_0^{\infty}\frac{dv}{[\cosh(r)\cosh(v) + \cosh(\rho)]^{n+1}} dv & = \int_0^{\infty} dy y^n e^{-y\cosh(\rho)}\int_0^{\infty} dv e^{-y\cosh(r)\cosh(v)} 
\\& = \int_0^{\infty} dy y^n e^{-y\cosh(\rho)}K_0(y\cosh(r)) 
\end{align*}
and by remembering \eqref{HW}. 

For the quaternionic hyperbolic space $H_n(\mathbb{H})$, analogues of \eqref{CH1} and \eqref{CH2}  were already proved in \cite{Mat} (see Theorem 5.1 there and its proof, see also \cite{Zhu}). This symmetric space is a AN group with parameters 
 $k = 3, m = 4(n-1)$ so that our formula reads: 
\begin{equation}\label{QuatHyp}
q_t^{(3,4(n-1))}(r) =  \frac{e^{-(2n+1)^2t/8}}{2^{3(4n-1)/2 + 1} \pi^{2n}[\cosh(r/2)]} \int_0^{\infty}  y^{2n-1} u\left(\frac{t}{4}, y\right)K_{1}\left(y\cosh(r/2)\right) dy,
\end{equation}
and may be related to the formula found in \cite{Mat} using similar arguments as above for $H_n(\mathbb{C})$.

\section{Concluding remarks}
Another integral representation of the heat kernel $q_t^{(k,m)}$ which is reminiscent of the semi-direct structure of a AN group and which has a probabilistic flavour is given by (\cite{Mus}): 
\begin{equation*}
q_t^{(k,m)}(z,\mathfrak{n}) = \frac{e^{-z^2/(2t)}}{\sqrt{2\pi t}}\int_0^{\infty} h_s^{(k,m)}(\mathfrak{n}) a_t(s,z) ds, 
\end{equation*}
where $\mathfrak{n}\in \textrm{Lie}(N), z \in \textrm{Lie}(A) \sim \mathbb{R}$, and $h_s^{(k,m)}(\mathfrak{n})$ is the heat kernel of $N$. An expression of the latter was derived in \cite{Ran}, Lemma 1.3.4, (see also \cite{Yang-Zhu}) and subsequently used in \cite{Cor} to write down $q_t^{(k,m)}(z,n)$ as a double integral involving a Bessel function of the first kind and an associated Legendre function. 

On the other hand, if we replace $u(t/4, y)$ in Proposition \ref{HeatNA} by its integral \eqref{HW} and use Fubini Theorem, then we are led to the integral:
\begin{equation}\label{IntegMac} 
 \int_0^{\infty} y^{(m+k-1)/2} e^{-y\cosh(\rho)} K_{(k-1)/2}\left(y\cosh\left(\frac{r}{2}\right)\right) dy. 
\end{equation}
From \cite{Erd2}, 7.7.3, (26), (see also \cite{Gra-Ryz}, p.700, (3)): 
\begin{equation}\label{McDonald}
\int_0^{\infty} y^{\mu-1}e^{-\alpha y} K_{\nu}(\beta y) dy = \frac{\sqrt{\pi}(2\beta)^{\nu}}{(\alpha + \beta)^{\mu+\nu}} \frac{\Gamma(\mu+\nu)\Gamma(\mu-\nu)}{\Gamma(\mu+1/2)} 
{}_2F_1\left(\mu+\nu, \nu + \frac{1}{2}, \mu+\frac{1}{2}; \frac{\alpha-\beta}{\alpha+\beta}\right)
\end{equation}
valid for $\Re(\mu) > |\Re(\nu)|, \Re(\alpha+\beta)>0$, \eqref{IntegMac} may be expressed through the Gauss hypergeometric function. Doing so leads to a single integral representation of $q_t^{(k,m)}(r)$, which may be seen as the analogue of Gruet's formula \eqref{GruetFor} for harmonic AN groups.

Finally, using Proposition \ref{HeatNA}, the Laplace transform (\cite{Jak-Wis}): 
\begin{equation*}
\int_0^{\infty}e^{-\lambda t}u(t,y)dt = {\it I}_{\sqrt{2\lambda}}(y), \quad \lambda > 0, 
\end{equation*} 
where ${\it I}_{\sqrt{2\lambda}}$ is the modified Bessel function of the first kind, and entry (5), p. 684 in \cite{Gra-Ryz}, one retrieves (up to normalisations) the Green function of a harmonic AN group displayed in \cite{ADY}, eq. (2.37) (see also \cite{Mat} for hyperbolic spaces over division algebras). In this respect, an interesting connection between Green functions of harmonic AN groups and odd dimensional real hyperbolic spaces was established in \cite{IOM}. 

{\bf Acknowledgments}: The author thanks Luc Del\'eaval for his valuable remarks.

\end{document}